\documentclass{amsart}
\usepackage{amsmath,amssymb,amsthm,url,comment}
\usepackage{graphicx}

\newtheorem{theorem}{Theorem}[section]  

\newtheorem{proposition}[theorem]{Proposition}

\newtheorem{corollary}[theorem]{Corollary}
\theoremstyle{definition}

\newtheorem*{remark*}{Remark}
\newcommand{\Z}{\mathbb{Z}}

 \makeatletter
    
    \@addtoreset{equation}{section}
  \makeatother

\newcommand{\CrossP}{
\begin{picture}(26,26)
\put(0,0){\vector(1,1){26}}
\put(10,16){\vector(-1,1){10}}
\put(26,0){\line(-1,1){10}}
\end{picture}  
}
\newcommand{\CrossN}{

\begin{picture}(26,26)
\put(0,0){\line(1,1){10}}
\put(16,16){\vector(1,1){10}}
\put(26,0){\vector(-1,1){26}}
\end{picture} 
}
\newcommand{\CrossO}{
\begin{picture}(26,26)
\qbezier(0,0)(14,12)(0,26)
\qbezier(26,0)(10,12)(26,26)
\put(25,25){\vector(1,1){1}}
\put(1,25){\vector(-1,1){1}}
\end{picture}
}

\begin{document}

\title[Gordian distance one and cosmetic crossing for genus one knots]{An obstruction of Gordian distance one and cosmetic crossings for genus one knots}
\author{Tetsuya Ito}
\address{Department of Mathematics, Kyoto University, Kyoto 606-8502, JAPAN}
\email{tetitoh@math.kyoto-u.ac.jp}

\begin{abstract}
We give an obstruction for genus one knots $K$, $K'$ to have the Gordian distance one by using the $0$th coefficient of the HOMFLT polynomials. As an application, we give a new constraint for genus one knot to admit a (generalized) cosmetic crossing. Combining known results, we prove the (generalized) cosmetic crossing conjecture for genus one pretzel knots.
\end{abstract}

\maketitle

\section{Introduction}

For oriented knots $K,K'$ in $S^{3}$, the \emph{Gordian distance} $d_G(K,K')$ of $K$ and $K'$ is the minimum number of crossing change needed to transform $K$ into $K'$. In particular, when $K'$ is the unknot $U$, $d_G(K,U)$ is nothing but the \emph {unknotting number}.

Let $P_K(v,z)$ be the HOMFLY polynomial of a knot or link $K$ in $S^{3}$, defined by the skein relation
\[ v^{-1} P_{K_+}(v,z) -  vP_{K_-}(v,z)= z P_{K_0}(v,z), \quad P_{\sf Unknot}(v,z) =1. \]
It is known that the HOMFLY polynomials are written in the form 
\[ P_K(v,z)=(v^{-1}z)^{-r_K + 1} \sum_{i=0} p^{i}_K(v)z^{2i}, \qquad p^{i}_K(v) \in \Z[v^2,v^{-2}]. \]
Here $r_K$ denotes the number of components of $K$. We call the polynomial $p^{i}_K(v)$ the \emph{$i$-th coefficient (HOMFLY) polynomial} of $K$.

The aim of this paper is point out the following obstruction for Gordian distance one for genus one knots. For a knot $K$, we denote by $a_2(K)$ the coefficient of $z^{2}$ for the Conway polynomial $\nabla_K(z)$ of $K$.
\begin{theorem}
\label{theorem:main}
Let $K$ and $K'$ be genus one knots. Assume that $K'$ is obtained from $K$ by the crossing change at a non-nugatory crossing with sign $\varepsilon = \pm 1$. 
Then 
\[ v^{-\varepsilon} p^0_K(v)- v^{\varepsilon}p^0_{K'}(v) = \varepsilon(v^{-1}-v)v^{2\varepsilon(a_2(K)-a_2(K'))}f(v)^{2} \]
for some $f(v) \in \Z[v^2,v^{-2}]$ such that $f(1)=1, f'(1)=0$.
\end{theorem}

Here a crossing $c$ of a diagram $D$ is \emph{nugatory} if there is a circle $C$ on the projection plane that transverse to the diagram $D$ only at $c$.

Obviously, the crossing change at a nugatory crossing does not change the knot.
A \emph{cosmetic crossing} is a non-nugatory crossing $c$ such that the crossing change at $c$ does not change the knot.
The \emph{cosmetic crossing conjecture} \cite[Problem 1.58]{ki} asserts such a crossing never exists.

The cosmetic crossing conjecture for genus one knots was studied in \cite{bfkp} where they proved the conjecture under several assumptions, such as $K$ is algebraically non-slice. In \cite{it} we extended the result for the case the Alexander polynomial of $K$ is non-trivial, and gave some criterion that can be applied for the case the Alexander polynomial is trivial.  

However, even for the genus one pretzel knot, the simplest family of genus one knots, there are infinitely many cases where all the known criteria fail to rule our cosmetic crossings. 

By putting $K=K'$, Theorem \ref{theorem:main} immediately yields a constraint for genus one knot to admit a cosmetic crossing; when a genus one knot $K$ admits a cosmetic crossings, then $p^0_K(v)=f(v)^{2}$ for some $f(v) \in \Z[v^2,v^{-2}]$ satisfying $f(1)=1, f'(1)=0$.

As we will discuss in Section \ref{section:general}, it turns out that this criterion actually applies to \emph{generalized} cosmetic crossings (see Section \ref{section:general} for a definition of generalized cosmetic crossing).

\begin{theorem}
\label{theorem:g-cosmetic}
Let $K$ be a knot of genus one. If $K$ admits a generalized cosmetic crossing then $p^0_K(v)=f(v)^{2}$ for some $f(v) \in \Z[v^2,v^{-2}]$ satisfying $f(1)=1, f'(1)=0$. 
\end{theorem}

By using this obstruction we finally confirm the (generalized) cosmetic crossing conjecture for genus one pretzel knots. 

\begin{theorem}
\label{theorem:genus-one-pretzel}
A genus one pretzel knot does not admit a generalized cosmetic crossing.
\end{theorem}

\section{Gordian distance one obstruction}

\subsection{Crossing change between two genus one knots}
\label{section:skein}
First we review a geometric content of crossing change, following \cite{st,kl}.

A \emph{crossing disk} $D$ for an oriented knot $K$ is an embedded disk such that whose interior intersects $K$ at exactly two points with opposite signs.
By suitably taking a crossing disk $D$, a crossing change at a crossing $c$ can be seen as $\varepsilon = \pm 1$ Dehn surgery on $\partial D$. Moreover, the crossing $c$ is nugatory if and only if $\partial D$ bounds a disk in the knot complement $S^{3}\setminus K$.

Let $K_+,K_-,K_0$ be the knots (or, diagrams) which are the same out side a small 3-ball, and in the small ball $K_+,K_-,K_{0}$ are $+1,-1,0$-tangle oriented as follows.\\
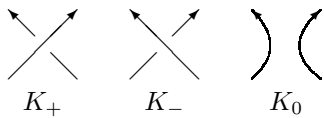
\[ \begin{array}{ccccccc}
\CrossP & & \CrossN & & \CrossO \\
K_+ &  &K_- & & K_0 & & \\
\end{array}\]
We call $(K_+,K_-,K_0)$ a \emph{skein triple} and we say that a skein triple is \emph{non-trivial} if the crossing in the 3-ball is non-nugatory.

For a link $L$, let $\chi(L)$ be the maximum euler characteristic of Seifert surface of $L$. Here we allow a non-connected Seifert surface, but we do not allow to have a closed components.

A key geometric ingredient of the proof of Theorem \ref{theorem:main} is the following proposition, which is a reformulation of \cite[Theorem 1.4]{st}, \cite[Theorem 2.1]{kl} (see also \cite[Proposition 2.2]{bfkp})
\begin{proposition}
\label{prop:skein-geometric}
Let $(K_{+},K_{-},K_0)$ be a non-trivial skein triple and assume that $K_+$ and $K_-$ are genus one knots. Then $K_0$ bounds an annulus.
\end{proposition}

\begin{proof}
Let $(K_+,K_-,K_0)$ be a non-trivial skein triple, and let $D$ be the crossing disk for $K_+$.
Since the linking number of $\partial D$ and $K_+$ is zero, $K_+$ bounds a surface $S$ in the complement $S^{3} \setminus \partial D$. We take a such a surface so that its genus is minimum. By \cite[Theorem 2.1]{kl}, $g(K_0) \leq g(S)=\max\{g(K_+),g(K_-)\}$. Therefore $g(S)=1$.

We put $D$ so that the intersection $\alpha:=D \cap S$ is a single arc. Since $c$ is assumed to be non-nugatory, $\alpha$ is essential in $S$. Thus by resolving the crossing $c$, we get a connected Seifert surface $S_0$ of $K_0$ such that $\chi(S_0)=\chi(S)+1=0$. Hence $K_0$ bounds an annulus $S_0$. 
\end{proof}

\subsection{Zeroth coefficient of the HOMFLY polynomial}
\label{section:homfly}
Let $\delta=\frac{1}{2}(r_{K_+} - r_{K_0}+ 1) \in \{0,1\}$. Namely, we define $\delta=0$ if two strands in the skein triple belongs to the same components, and $\delta=1$ otherwise.
From the skein relation, the zeroth coefficient polynomial of $p^{0}_K(v)$  satisfies the following simple but remarkable skein relations.
\[ v^{-2} p^0_{K_+}(v) - p^{0}_{K_-}(v)= 
\begin{cases} p^0_{K_0}(v) & \mbox{ if } \delta=0, \\
0 & \mbox{ otherwise}. 
\end{cases} \] 

Consequently, the zeroth coefficient polynomial $p^{0}_{L}(v)$ of a link $L$ is determined by the zeroth coefficient polynomial $p^{0}_{K_i}(v)$ of each component $K_i$ and the total linking number $lk(L)=\sum_{i<j}lk(K_i,K_j)$.

\begin{proposition}
Let $L=K_1 \cup \cdots \cup K_n$ be an $n$-component link. Then 
\[ p^{0}_{L}(v)= (v^{-2}-1)^{n-1}v^{2lk(L)} p^{0}_{K_1}(v)p^{0}_{K_2}(v)\cdots p^{0}_{K_n}(v)\]
\end{proposition} 

The zeroth coefficient polynomial $p^{0}_K(v)$ of a knot $K$ satisfies $p^{0}_K(1)=1$ and $(p^{0}_K)'(1)=0$. Conversely, every $f(v) \in \Z[v^{2},v^{-2}]$ satisfying $f(1)=1$ and $f'(1)=0$ is realized as the zeroth coefficient polynomial of some knots \cite{ka}.

Summarizing, when a two-component link $K$ bounds an annulus, the zeroth coefficient polynomial $p^{0}_K(v)$ takes a special form.
\begin{corollary}
\label{cor:annulus}
Let $L$ be a two-component link that bounds an annulus.
Then $p^{0}_K(v)=(v^{-2}-1)v^{2lk(L)}f(v)^2$ for some $f(v) \in \Z[v^2,v^{-2}]$ satisfying $f(1)=1$ and $f'(1)=0$.
\end{corollary}

\subsection{Proof of Theorem \ref{theorem:main}}

Theorem \ref{theorem:main} is obtained by combining the arguments in Section \ref{section:skein} and \ref{section:homfly}.

\begin{proof}[Proof of Theorem \ref{theorem:main}]
We prove the theorem for the case $\varepsilon=+1$. The case $\varepsilon=-1$ is similar.
Let $(K_+=K,K_-=K',K_0)$ be the non-trivial skein triple.
By Proposition \ref{prop:skein-geometric}, $K_0$ bounds an annulus hence by Corollary \ref{cor:annulus}, $p^{0}_{K_0}(v)=(v^{-2}-1)v^{2lk(K_0)}f(v)^{2}$ for some $f(v) \in \Z[v^2,v^{-2}]$ satisfying $f(1)=1$ and $f'(1)=0$.

By the skein relation of the Conway polynomial, $lk(K_0)=a_2(K_+)-a_2(K_-)$. Hence 
 by the skein relation of the $0$-th coefficient polynomial we conclude that
\[ v^{-2}p^{0}_{K_+}(v)- p^0_{K_-}(v) = (v^{-2}-1)v^{2(a_2(K_+)-a_2(K_-))}f(v)^2.\]

\end{proof}

\section{Generalized cosmetic crossings}
\label{section:general}

A generalized crossing change (of degree $q$) is the $\frac{1}{q}$ surgery ($q \in \Z \setminus \{0\}$) on a crossing disk (note that $q=\pm 1$ is the usual crossing change). A \emph{generalized cosmetic crossing} is a non-nugatory crossing such that a generalized crossing change at the crossing yields the same knot.

As a natural generalization of the cosmetic crossing conjecture, it is conjectured that no knot admits a generalized cosmetic crossings. Indeed, in most literatures on cosmetic crossings including early ones \cite{to,ka} and recent ones \cite{ro,wa}, it is actually discussed and studied generalized cosmetic crossings and `cosmetic surgery conjecture' is often used to represent this stronger version.

As long as the author knows, all the assertions concerning cosmetic crossings actually can be applied for generalized cosmetic crossings, with some modifications or additional arguments, if necessary.

For example, in the author's previous paper \cite{it} or \cite{bfkp} on genus one knots, they only explicitly stated and discussed usual cosmetic crossings. However, one can check that a similar argument actually shows the non-existence of generalized cosmetic crossings. In particular, it is actually proved that a genus one knot with non-trivial Alexander polynomial does not admit cosmetic crossings.

The same is true for the condition $p^{0}_K(v)=f(v)^{2}$, even though it appeared as a special case of the Gordian distance one obstruction.

\begin{proof}[Proof of Theorem \ref{theorem:g-cosmetic}]
Assume that generalized crossing change of degree $q$ at a non-nugatory crossing $c$ produces the same knot. We treat the case $q>0$. For each $i=0,1,\ldots,q-1$, let $K_{i}$ be the knot obtained by the generalized crossing change of degree $i$ at the crossing $c$ (here $K_0=K$), and let $(K_{i},K_{i+1},K_{i,0})$ be the corresponding skein triple. By definition, $K_{0,0}=K_{1,0}=\cdots = K_{q-1,0}$, so we put this link $L$. The same argument as Proposition \ref{prop:skein-geometric} shows that $L$ bounds an annulus so $p^{0}_{L}(v)=(v^{-2}-1)v^{2lk(L)}f(v)^{2}$ for some $f(v) \in \Z[v^2,v^{-2}]$ satisfying $f(1)=1$ and $f'(1)=0$. Since
\begin{align*}
0 &= a_{2}(K_0) -a_{2}(K_q) = \sum_{i=0}^{q-1}(a_{2}(K_i)-a_{2}(K_{i+1})) = \sum_{i=0}^{q-1}lk(K_{L}) = q lk(L)
\end{align*}
$lk(L)=0$. Then 
\begin{align*}
(v^{-2}-v^{2q-2})p^{0}_{K}(v)& = v^{-2}p^{0}_{K_0}(v) - v^{2q-2}p^{0}_{K_{q}}(v)\\
& = \sum_{i=0}^{q-1} v^{2i}(v^{-2}p^{0}_{K_i}(v)  - p^{0}_{K_{i+1}}(v))\\
& = \sum_{i=0}^{q-1} v^{2i}p^{0}_{L}(v) = \sum_{i=0}^{q-1} v^{2i}(v^{-2}-1)f(v)^{2}\\
& = (v^{-2}-v^{2q-2})f(v)^{2}
\end{align*}
hence $p^{0}_{K}(v)=f(v)^{2}$.
\end{proof}

\begin{proof}[Proof of Theorem \ref{theorem:genus-one-pretzel}]
Let $K=P(p,q,r)$ be a genus one pretzel knot where $p,q,r$ are odd integers.
Then $a_{2}(K)=\frac{pq+qr+pr+1}{4}$. 

As we have mentioned, since we have proven the generalized cosmetic crossing conjecture when the Alexander polynomial of $K$ is non-trivial \cite{it}, we assume that $a_2(K) = 0$. By taking the mirror image if necessary, so it is sufficient to treat the case $p\geq q >0 >r$.
Note that $r\neq -1$ because, when $r=-1$, $a_2(K)=\frac{pq+qr+pr+1}{4}=0$ implies $q=1$ so $K$ is the trivial knot. By a similar reason, $q\neq 1$ as well. 

The zeroth coefficient polynomial of $K=P(p,q,r)$ is computed from the skein relation, and given by
\[ p^0_K(v)= v^{p+q}-v^{p+q+r+1}-v^{p+q+r-1}+v^{p+r}+v^{q+r} \]
(see \cite[Proposition 2.2 (i)]{ta}, for example).
Since $p\geq q >1 > -1 > r$ 
\[ p+q > p+q+r+1 > p+q+r-1 > p+r \geq q+r \]
If $p^{0}_K(v)=f(v)^{2}$, then the coefficient of $v^{p+q+r+1}$ should be even so $p^0_K(v)$ is not the square of other polynomials.
Therefore by Theorem \ref{theorem:g-cosmetic} $K=P(p,q,r)$ does not admit a generalized cosmetic crossing.
\end{proof}

\section*{Acknowledgement}
The author has been partially supported by JSPS KAKENHI Grant Number 19K03490 and 21H04428.  
He would like to thank H. Takioka for useful conversations, in particular, for 
informing me of his computation of the zeroth coefficient polynomial of genus one pretzel knots.

\end{document}